\documentclass[12pt]{amsart}

\usepackage{amsfonts}
\usepackage{amsmath}
\usepackage{amssymb}
\usepackage{amsthm}
\usepackage{bm}
\usepackage{enumerate} 
\usepackage[margin=1.2in]{geometry}
\usepackage{hyperref}

\theoremstyle{plain}
\newtheorem{theorem}{Theorem}[section]
\newtheorem{lemma}[theorem]{Lemma}
\newtheorem{proposition}[theorem]{Proposition}

\theoremstyle{definition}
\newtheorem{definition}[theorem]{Definition}

\newtheorem{remark}[theorem]{Remark}

\newtheorem{example}[theorem]{Example}

\numberwithin{equation}{section}

\newcommand\N{\mathbb{N}}
\newcommand\Z{\mathbb{Z}}
\newcommand\R{\mathbb{R}}
\newcommand\C{\mathbb{C}}

\renewcommand\S{\mathcal{S}}

\newcommand\M{\mathfrak{M}}
\newcommand\NN{\mathfrak{N}}

\newcommand\V{\mathcal{V}}
\newcommand\W{\mathcal{W}}

\DeclareMathOperator\supp{supp}

\DeclareMathOperator\loc{loc}

\DeclareMathOperator\condL{L}
\DeclareMathOperator\condwI{wI}
\DeclareMathOperator\condI{I}

\DeclareMathOperator\condwM{wM}
\DeclareMathOperator\condM{M}

\newcommand\jb[1]{\langle#1\rangle}

\begin{document}

\title[On the inclusion relations between Gelfand-Shilov spaces]{On the inclusion relations between Gelfand-Shilov spaces}

\author[A. Debrouwere]{Andreas Debrouwere}
\thanks{}
\address{Department of Mathematics and Data Science\\Vrije Universiteit Brussel\\ Pleinlaan 2\\1050 Brussels\\Belgium}
\email{andreas.debrouwere@vub.be}

\author[L. Neyt]{Lenny Neyt}
\address{University of Vienna\\ Faculty of Mathematics\\ Oskar-Morgenstern-Platz 1 \\ 1090 Wien\\ Austria}
\thanks{L. Neyt gratefully acknowledges support by the Alexander von Humboldt Foundation, the Research Foundation -- Flanders through the postdoctoral grant 12ZG921N, and the Austrian Science Fund (FWF) through the project 10.55776/ESP8128624}
\email{lenny.neyt@univie.ac.at}

\author[J. Vindas]{Jasson Vindas}

\address{Department of Mathematics: Analysis, Logic and Discrete Mathematics\\ Ghent University\\ Krijgslaan 281\\ 9000 Gent\\ Belgium}
\email{jasson.vindas@UGent.be}
\thanks{J. Vindas was supported by the Research Foundation -- Flanders through the grant G067621N and Ghent University through the grant bof/baf/4y/2024/01/155}


\begin{abstract}
We study inclusion relations between Gelfand-Shilov type spaces defined via a weight (multi-)sequence system, a weight function system, and a translation-invariant Banach function space. 
We characterize when such spaces are included into one another in terms of growth relations for the defining weight sequence  and weight function systems.
Our general framework allows for a unified treatment of the Gelfand-Shilov spaces $\mathcal{S}^{[M]}_{[A]}$  (defined via weight sequences $M$ and $A$) and the Beurling-Bj\"orck spaces $\mathcal{S}^{[\omega]}_{[\eta]}$ (defined via weight functions $\omega$ and $\eta$). 
\end{abstract}

\maketitle

\section{Introduction}

The problem of characterizing inclusion relations between ultradifferentiable classes goes back to a question of Carleman \cite{Carleman} and was, among others, thoroughly studied by Mandelbrojt (see Chapitre VI of his thesis \cite{Mandelbrojt} and the references therein). We refer to \cite{Esser, JSS, N-S, R-S-CompUltradiffClass} for recent works related to this topic.

The goal of this article is to obtain characterizations of inclusion relations between Gelfand-Shilov type spaces \cite{G-S} (= weighted spaces of ultradifferentiable functions defined on the whole of $\R^n$). These 
spaces, also known as spaces of type $\mathcal{S}$, have been intensively studied over the past few years, see e.g.\  \cite{B-J-O-S2,  C-G-P-R, C-T, D-N-V-NuclGSSpKernThm, P-P-V}. The study of inclusion relations in this setting was recently initiated by Boiti et al. \cite{BJOS-inc,BJOS-lc}.

 We shall work here with a novel broad class of Gelfand-Shilov spaces. Namely, our spaces are defined through a multi-indexed weight sequence system \cite{D-N-V-NuclGSSpKernThm,  R-S-CompUltradiffClass} (also sometimes called a weight matrix), a weight function system \cite{D-N-V-NuclGSSpKernThm}, and a translation-invariant Banach function space (cf.\ \cite{F-G-BanachSpIntGroupRepAtomicDecompI}) (generalizing the Lebesgue spaces $L^p(\R^n)$, $p \in [1, \infty]$).
This general framework leads to a unified treatment of Gelfand-Shilov spaces defined via weight sequences \cite{Komatsu} or weight functions \cite{B-M-T-UltradiffFuncFourierAnal}, and via different $L^p$-norms.
 Moreover, as we consider multi-indexed weight sequence systems, our results  cover the anisotropic case as well.
 
The main difference between \cite{BJOS-inc,BJOS-lc}  and our work is that in \cite{BJOS-inc, BJOS-lc} only Fourier invariant Gelfand-Shilov spaces are considered for which the defining regularity and decay conditions are quantified in terms of a single weight sequence system,
whereas we will consider spaces that are not necessarily Fourier-invariant for which the defining regularity and decay conditions are quantified separately in terms of a weight sequence system and a weight function system, respectively. We refer to Remark \ref{r:ComparisonBJOS} for a more detailed comparison between our spaces and the ones considered in  \cite{BJOS-inc, BJOS-lc}. Furthermore, our proof methods are completely different from the ones used in \cite{BJOS-inc, BJOS-lc}.

We now state two important samples of our results. Firstly, we consider Gelfand-Shilov spaces $\mathcal{S}^{(M)}_{(A),p}$ (Beurling case) and $\mathcal{S}^{\{M\}}_{\{A\},p}$ (Roumieu case), $p   \in [1,\infty]$, defined via single isotropic weight sequences $M = (M_q)_{q \in \N}$ and $A = (A_q)_{q \in \N}$, and the $L^p$-norm. We refer to Sections \ref{s:prel} and \ref{sec:MainResults} for the precise definition of these spaces. We will use  $\mathcal{S}^{[M]}_{[A],p}$ as a common notation for $\mathcal{S}^{(M)}_{(A),p}$ and $\mathcal{S}^{\{M\}}_{\{A\},p}$; a similar convention will be used for other spaces and notations. 
	\begin{theorem}
		\label{t:InclusionCharClassicalGS}
		Let $p \in [1, \infty]$. Let $M, N, A, B$ be isotropic weight sequences. Suppose that $M$ and $A$ are log-convex and $\mathcal{S}^{[M]}_{[A],p} \neq \{0\}$. The following statements are equivalent:
			\begin{itemize}
				\item[(i)] $M \preceq N$ and $A \preceq B$, i.e., there are $C,H >0$ such that
				$$
				M_q \leq CH^q N_q, \quad \mbox{and} \quad A_q \leq CH^q B_q, \qquad q \in \N.
				$$
				\item[(ii)] $\S^{[M]}_{[A], p} \subseteq \S^{[N]}_{[B], p}$ as sets.
				\item[(iii)] $\S^{[M]}_{[A], p} \subseteq \S^{[N]}_{[B], p}$ continuously.
			\end{itemize}
	\end{theorem}
Next, we consider Beurling-Bj\"orck spaces $\mathcal{S}^{[\omega]}_{[\eta],p}$ \cite{bjorck66} where $\omega$ is a Braun-Meise-Taylor weight function \cite{B-M-T-UltradiffFuncFourierAnal}, $\eta: [0, \infty) \to [0, \infty)$ is a non-decreasing continuous function, 
  and $p \in [1, \infty]$. Again, we refer to Sections \ref{s:prel} and \ref{sec:MainResults} for the precise definition of these spaces. 	

	\begin{theorem}
		\label{t:InclusionCharBMT}
		Let $p \in  [1, \infty]$. Let $\omega, \sigma$ be Braun-Meise-Taylor weight functions and let  $\eta, \rho: [0, \infty) \to [0, \infty)$ be non-decreasing continuous functions such that $\rho$ is unbounded. 
		Suppose that $\mathcal{S}^{[\omega]}_{[\eta],p} \neq \{0\}$.  The following statements are equivalent:
			\begin{itemize}
				\item[(i)] $\sigma(t) = O(\omega(t))$ and $\rho(t) = O(\eta(t))$.
				\item[(ii)] $\S^{[\omega]}_{[\eta], p} \subseteq \S^{[\sigma]}_{[\rho], p}$ as sets.
				\item[(iii)] $\S^{[\omega]}_{[\eta], p} \subseteq \S^{[\sigma]}_{[\rho], p}$ continuously.
			\end{itemize}
	\end{theorem}	
This article is organized as follows. In the preliminary Section \ref{s:prel} we introduce the necessary notions to define the Gelfand-Shilov type spaces that we will be concerned with in this article. Our main results are stated in Section \ref{sec:MainResults}. Finally, the proofs of these results are given in Section \ref{s:proofs}.

\section{Preliminaries}\label{s:prel}

In this preliminary section, we introduce and discuss translation-invariant Banach function spaces, weight function systems, and weight sequence systems. These notions will be used in the next section to define the Gelfand-Shilov type spaces that we shall work with in this article.
We denote translation by $x \in \R^{n}$ as $T_{x} f(t) = f(t - x)$ and reflection about the origin as $\check{f}(t) = f(-t)$. We write $1_A$ for the indicator function of a set $A \subseteq \R^n$. 
\subsection{Banach function spaces}
The following definition is much inspired by the Banach function spaces used in the coorbit theory of Feichtinger and Gr\"{o}chening \cite{F-G-BanachSpIntGroupRepAtomicDecompI} (cf. \cite[Sections 7 and 8]{D-H-V}).

	\begin{definition}
		A Banach space $E$ is called a \emph{translation-invariant Banach function space (TIBF) of bounded type \footnote{\emph{Bounded type} essentially refers to property (A.2).}} on $\R^n$ if $E$ is non-trivial, the continuous inclusion $E \subseteq  L^{1}_{\loc}(\R^n)$ holds, and $E$ satisfies the following three conditions:
			\begin{itemize}
				\item[(A.1)] $T_{x} E \subseteq E$ for all $x \in \R^{n}$.
				\item[(A.2)] There exists $C_0 > 0$ such that $\|T_{x} f\|_{E} \leq C_0 \|f\|_{E}$ for all $x \in \R^{n}$ and $f \in E$.
				\item[(A.3)] 
				$E * C_{c}(\R^{n}) \subseteq E$. 
			\end{itemize}
		The Banach space $E$ is called \emph{solid} if for all $f \in E$ and $g \in L^{1}_{\loc}(\R^n)$ we have that
			\[ |g(x)| \leq |f(x)| \text{ for almost all } x \in \R^n \quad \Longrightarrow \quad g \in E \text{ and } \|g\|_{E} \leq \|f\|_{E} . \] 
	\end{definition}
	
	\begin{remark}
		Let $E$ be a solid TIBF of bounded type. Then, 
		every element of $L^\infty(\R^n)$ with compact support belongs to $E$ (cf.\ \cite[Lemma 3.9]{F-G-BanachSpIntGroupRepAtomicDecompI}). In particular, $1_{K} \in E$ for every compact $K \subseteq \R^{n}$.
	\end{remark}
	
	\begin{example} (i) The Lebesgue spaces $L^p = L^p(\R^n)$, $p \in [1,\infty]$, are solid TIBF of bounded type on $\R^n$. We define $L^0 = L^0(\R^n)$ as the space consisting of all $f \in L^\infty$ such that for every $\varepsilon >0$ there is a compact $K \subseteq \R^n$ such that $|f(x)| \leq \varepsilon$ for almost all $x \in \R^n \backslash K$. We endow $L^0$ with the subspace topology induced by  $L^\infty$. Then, $L^0$ is a solid TIBF of bounded type on $\R^n$.
		
		\noindent	
		(ii) The mixed-norm Lebesgue spaces $L^{p_1,p_2}(\R^{n_1+n_2}) = L^{p_1}(\R^{n_1}; L^{p_2}(\R^{n_2}))$, $p_1,p_2 \in [1,\infty]$, are solid TIBF of bounded type on $\R^{n_1+n_2}$.
	
	\end{example}
	
Following \cite[Definition 3.4]{F-G-BanachSpIntGroupRepAtomicDecompI}, we associate a Banach sequence space with each solid TIBF of bounded type in the following way.
	
	\begin{definition}
		\label{def:Ed}
		Let $E$ be a solid TIBF of bounded type. We define the space $E_d$ as the space consisting of all $c = (c_{j})_{j \in \Z^{n}} \in \C^{\Z^{n}}$ such that 
			\[ \sum_{j \in \Z^{n}} c_{j} T_{j} 1_{[0, 1]^{n}} \in E \]
		and endow it with the norm $\|c\|_{E_{d}} = \|\sum_{j \in \Z^{n}} |c_{j}| T_{j} 1_{[0, 1]^{n}}\|_{E}$. Then, $E_{d}$ is a Banach space.
	\end{definition}
	
	\begin{remark}
		\label{remark-inclusion}
		Let $E$ be a solid TIBF of bounded type. Then, $\ell^1 \subseteq E_{d} \subseteq \ell^\infty$  continuously (cf.\ \cite[Lemma 3.5(a)]{F-G-BanachSpIntGroupRepAtomicDecompI}).
	\end{remark}
	
	\begin{example} 
		(i) $L^p(\R^n)_d = \ell^p(\Z^n)$, $p \in [1,\infty]$, and, $L^0(\R^n)_d = c_0(\Z^n)$. \\
		(ii) $L^{p_1,p_2}(\R^{n_1+n_2})_d = \ell^{p_1,p_2}(\Z^{n_1+n_2})= \ell^{p_1}(\Z^{n_1}; \ell^{p_2}(\Z^{n_2}))$, $p_1,p_2 \in [1,\infty]$.
	\end{example}

The following two results will be used later on.
	
	\begin{lemma}[cf.\ {\cite[Proposition 5.1]{F-G-BanachSpIntGroupRepAtomicDecompI}}]
		\label{l:ConvEval}
		Let $E$ be a solid TIBF of bounded type. For all $f \in E$ and $\chi \in C_c(\R^n)$ it holds that 
			\[  S_{\chi}(f) = (f * \chi (j))_{j \in \Z^{n}} \in E_d. \]
	\end{lemma}
	
	\begin{proof}
		Let $f \in E$ and $\chi \in  C_{c}(\R^{n})$ be arbitrary. Since $E \subseteq L^{1}_{\loc}(\R^n)$, $f \ast \chi \in C(\R^n)$. Hence, the point values $f * \chi (j)$, $j \in \Z^{n}$, are well-defined. 
		 Choose  $\psi \in C_{c}(\R^{n})$ such that $\psi \equiv 1$ on $\supp \chi + [0, 1]^{n}$. Then, for all $j \in \Z^n$ and $x \in j + [0,1]^n$ it holds that
			 \[ f \ast \chi(j) = \int_{\R^n} f(t) \psi(x-t) \chi(j-t) dt. \] 
		 Hence, we obtain that for all $x \in \R^n$
			\[ \left| \sum_{j \in \Z^{n}} f * \chi(j) T_{j} 1_{[0, 1]^{n}}(x) \right| \leq \| \chi\|_{L^\infty} (|f| \ast |\psi|) (x). \]
		The result now follows from $E * C_{c}(\R^{n}) \subseteq E$ and the fact that $E$ is solid.
	\end{proof}
	
As customary, we write $\jb{x} = (1 + |x|^{2})^{1/2}$, $x \in \R^n$. We define $C_{\jb{ \, \cdot \,}^{n+1}} = C_{\jb{ \, \cdot \,}^{n+1}}(\R^n)$ as the Banach space consisting of all $f \in C(\R^n)$ such that
	\[ \| f\|_{\jb{ \, \cdot \,}^{n+1}} = \sup_{x\in \R^n} |f(x)| \jb{x}^{n+1} < \infty. \]
	
	\begin{lemma}[cf.\ {\cite[Proposition 5.2]{F-G-BanachSpIntGroupRepAtomicDecompI}}]
		\label{l:EdToE}
		Let $E$ be a solid TIBF of bounded type. 
		The bilinear mapping
			\[ E_{d} \times C_{\jb{ \, \cdot \,}^{n+1}} \to E, \quad (c,\psi) \mapsto R_{\psi}(c) = \sum_{j \in \Z^n} c_{j} T_{j} \psi, \]
		is well-defined and continuous. 
	\end{lemma}
	
	\begin{proof}
	Let $c \in E_d$ and $\psi \in C_{\jb{ \, \cdot \,}^{n+1}}$ be arbitrary. 
	 For all $k \in \Z^n$ it holds that
		$$
		\left |\sum_{j \in \Z^n} c_j T_{j-k} \psi T_j 1_{[0, 1]^{n}} \right| \leq \| \psi T_k 1_{[0, 1]^{n}} \|_{L^\infty} \sum_{j \in \Z^n} |c_j| T_j 1_{[0, 1]^{n}}.
		$$
		As $E$ is solid, we obtain that $\sum_{j \in \Z^n} c_j T_{j-k} \psi T_j 1_{[0, 1]^{n}} \in E$ 
		and
		$$
		 \left \| \sum_{j \in \Z^n} c_j T_{j-k} \psi T_j 1_{[0, 1]^{n}}\right \|_E \leq \| \psi T_k 1_{[0, 1]^{n}} \|_{L^\infty} \| c\|_{E_d}.
		 $$
		 Consequently,
		\begin{align*}
		  \sum_{k \in \Z^n} \left  \| T_k \left ( \sum_{j \in \Z^n} c_j T_{j-k} \psi T_j 1_{[0, 1]^{n}} \right) \right \|_E &\leq C_0  \| c\|_{E_d}   \sum_{k \in \Z^n} \| \psi T_k 1_{[0, 1]^{n}} \|_{L^\infty}  \\
		  &\leq C_0C  \| c\|_{E_d} \| \psi\|_{\jb{ \, \cdot \,}^{n+1}}, 
		\end{align*}
		 where $C_0$ is the constant from condition (A.2) and $C = (2(n+1))^{(n + 1) / 2} \sum_{k \in \Z^{n}} \jb{k}^{-(n + 1)}$. 
		 Since  we have the pointwise equality 
		 $$
		 R_{\psi}(c)
		 = \sum_{k \in \Z^n} T_k \left ( \sum_{j \in \Z^n} c_j T_{j-k} \psi T_j 1_{[0, 1]^{n}} \right),
		 $$
		and $E$ is solid and complete, we may conclude that $R_{\psi}(c) \in E$ and
		$$
		\| R_{\psi}(c)\|_E
		 \leq C_0C  \| c\|_{E_d} \| \psi\|_{\jb{ \, \cdot \,}^{n+1}} .
		$$
	\end{proof}

\subsection{Weight function systems}

By a \emph{weight function}, we mean a real-valued continuous function $w$ on $\R^n$ such that $w(x) \geq 1$ for all $x \in \R^n$. 
Following \cite{D-N-V-NuclGSSpKernThm}, a family $\W = \{ w^{\lambda} \mid \lambda \in \R_{+} \}$  of weight functions is called a \emph{weight function system} if $w^{\lambda}(x) \leq w^{\mu}(x)$ for all $x \in \R^n$ and $\mu \leq \lambda$.
Let $B(0, R) = \{ x \in \R^n \mid |x| < R \}$ for $R >0$.
We consider the following conditions on a weight function system $\W$:
	\begin{itemize}
		\item[$(\condwM)$] $\forall \lambda \in \R_{+} \, \exists \mu \in \R_{+} \, \exists C > 0\,  \forall x \in \R^n, y \in B(0,1) \, : \, w^{\lambda}(x + y) \leq C w^{\mu}(x)$.
		\item[$\{\condwM\}$] $\forall \mu \in \R_{+}  \,\exists \lambda \in \R_{+} \, \exists C > 0\,  \forall x \in \R^n, y \in B(0,1) \, : \, w^{\lambda}(x + y) \leq C w^{\mu}(x)$.
		\item[$(\condM)$] $\forall \lambda \in \R_{+} \, \exists \mu, \nu \in \R_{+} \, \exists C > 0 \, \forall x, y \in \R^n \, : \, w^{\lambda}(x + y) \leq C w^{\mu}(x) w^{\nu}(y)$.
		\item[$\{\condM\}$] $\forall \mu, \nu \in \R_{+} \, \exists \lambda \in \R_{+} \, \exists C > 0 \, \forall x, y \in \R^n \, : \, w^{\lambda}(x + y) \leq C w^{\mu}(x) w^{\nu}(y)$.
	\end{itemize}
Note that $[\condM]$ implies $[\condwM]$. (We recall again that we employ $[ \: ]$ as a common notation for treating both symbols $(\:)$ and $\{\:\}$ simultaneously.)

For two weight function systems $\W$ and $\V$ we write
	\begin{align*}
		\W (\subseteq) \V & \quad \Longleftrightarrow \quad \forall \lambda \in \R_{+} \, \exists \mu \in \R_{+} \, : \, v^{\lambda}(t) = O(w^\mu(t)), \\
		\W \{\subseteq\} \V & \quad \Longleftrightarrow \quad \forall \mu \in \R_{+} \, \exists \lambda \in \R_{+} \, : \, v^{\lambda}(t) = O(w^\mu(t)). 	
	\end{align*}

\subsection{Weight sequence systems}\label{s:wss}

A sequence $M = (M_{\alpha})_{\alpha \in \N^{n}}$ of positive numbers is called a \emph{weight sequence} if $M_0 = 1$ and $\lim_{|\alpha| \to \infty} (M_{\alpha})^{1 / |\alpha|} = \infty$.
We define its \emph{associated function}  as
	\[ \omega_{M}(x) = \sup_{\alpha \in \N^{n}} \log \frac{|x^{\alpha}|}{M_{\alpha}} , \qquad x \in \R^{n} . \]
Note that $\exp \omega_{M}$ is a weight function.

A weight sequence $M$ is said to be \emph{log-convex} if there exists a convex function $F : [0, \infty)^{n} \to \R$ with $F(\alpha) = \log M_{\alpha}$ for all $\alpha \in \N^{n}$.
In \cite[Section 5]{BJOS-lc} it was shown that the log-convex minorant $M^{\text{lc}} =  (M^{\text{lc}}_{\alpha})_{\alpha \in \N^{n}}$ of $M$ is given by
	\begin{equation}
		\label{eq:LogConvexMinorant}
		M^{\text{lc}}_{\alpha} = \sup_{x \in \R^{n}} \frac{|x^{\alpha}|}{\exp \omega_{M}(x)} , \qquad \alpha \in \N^{n} , 
	\end{equation}
i.e., $M^{\text{lc}}$ is the largest log-convex weight sequence such that $M^{\text{lc}}_{\alpha} \leq M_{\alpha}$ for all $\alpha \in \N^n$.
In particular, $M = M^{\text{lc}}$ if and only if $M$ is log-convex.
A weight sequence $M$ is called  \emph{isotropic} if $M = (M_{|\alpha|})_{\alpha\in \mathbb{N}^{n}}$ for a sequence $(M_{q})_{q\in\mathbb{N}}$.
\begin{remark} 
Let $e_j$ be the standard coordinate unit vectors in $\R^n$, $j = 1, . . . , n$. If a weight sequence $M$ is log-convex, then it satisfies
	$$
		M^2_{\alpha+e_j} \leq M_{\alpha} M_{\alpha + 2e_j} , \qquad \alpha \in \N^n, ~ j = 1, \ldots, n. 
	$$
The converse is true if $M$ is isotropic, but 
false for general weight sequences \cite[Example 5.4]{BJOS-lc}.  
\end{remark}

For two weight sequences $M$ and $N$ we define
	\[ M \subseteq N \quad \Longleftrightarrow \quad \exists C > 0 \, \forall \alpha \in \N^{n} \, : \, M_{\alpha} \leq C N_{\alpha} . \]
As in Theorem \ref{t:InclusionCharClassicalGS}, we write
	\[ M \preceq N \quad \Longleftrightarrow \quad \exists C,H > 0 \, \forall \alpha \in \N^{n} \, : \, M_{\alpha} \leq CH^{|\alpha|} N_{\alpha} . \]

A family $\M = \{ M^{\lambda} \mid \lambda \in \R_{+} \}$ of weight sequences is called a \emph{weight sequence system} if $M^{\lambda}_{\alpha} \leq M^{\mu}_{\alpha}$ for all $\alpha \in \N^{n}$ and $\lambda \leq \mu$. 
We call $\M$ log-convex  if each $M^{\lambda}$ is log-convex. We consider the following conditions on $\M$:
	\begin{itemize}
		\item[$(\condL)$] $\forall R > 0 \, \forall \lambda \in \R_{+} \, \exists \mu \in \R_{+} \, \exists C > 0 \, \forall \alpha \in \N^{n} \, : \, R^{|\alpha|} M^{\mu}_{\alpha} \leq C M^{\lambda}_{\alpha}$.
		\item[$\{\condL\}$] $\forall R > 0 \, \forall \mu \in \R_{+} \, \exists \lambda \in \R_{+} \, \exists C > 0 \, \forall \alpha \in \N^{n} \, : \, R^{|\alpha|} M^{\mu}_{\alpha} \leq C M^{\lambda}_{\alpha}$.
		\item[$(\condwI)$] $\forall \lambda \in \R_{+} \, \exists \mu \in \R_{+} \, \exists H > 0 \, \forall R > 0  \, \exists C > 0  \,\forall \alpha,\beta \in \N^{n} \, : \,  \newline M^{\mu}_{\alpha}R^{|\beta|} \leq CH^{|\alpha + \beta|}M^{\lambda}_{\alpha + \beta }$.
		\item[$\{\condwI\}$] $\forall \mu \in \R_{+} \, \exists \lambda \in \R_{+} \, \exists H > 0 \, \forall R > 0  \, \exists C > 0  \, \forall \alpha, \beta \in \N^{n} \, : \, \newline M^{\mu}_{\alpha} R^{|\beta|} \leq CH^{|\alpha + \beta|}M^{\lambda}_{\alpha + \beta }$.	
		\item[$(\condI)$] $\forall \lambda \in \R_{+} \, \exists \mu, \nu \in \R_{+} \, \exists C,H > 0 \, \forall \alpha,\beta \in \N^{n} \, : \,  M^{\mu}_{\alpha}M^{\nu}_{\beta} \leq CH^{|\alpha + \beta|}M^{\lambda}_{\alpha + \beta }$.
		\item[$\{\condI\}$] $\forall \mu, \nu \in \R_{+} \, \exists \lambda \in \R_{+} \, \exists C,H > 0 \, \forall \alpha,\beta \in \N^{n} \, : \,  M^{\mu}_{\alpha}M^{\nu}_{\beta}\leq CH^{|\alpha + \beta|}M^{\lambda}_{\alpha + \beta }$.
		\end{itemize}
Note that $[\condI]$ implies $[\condwI]$. The condition $[\condI]$  was introduced in \cite[Section 6]{BJOS-lc}. Given a single weight sequence $M$, we  define $\M_{M} = \{ (\lambda^{|\alpha|} M_{\alpha})_{\alpha \in \N^{n}} \, | \, \lambda \in \R_{+} \}$. Then, $\M_M$ is log-convex if and only if $M$ is so and $\M_M$ always satisfies $[\condL]$.

\begin{remark}
\label{r:Counterexample}
Every isotropic log-convex weight sequence $M$ satisfies
	\[ M_\alpha M_\beta \leq M_{\alpha+ \beta}, \qquad \alpha, \beta \in \N^d. \]
In particular, every weight sequence system consisting of isotropic log-convex weight sequences satisfies $[\condI]$. However, there exist log-convex weight sequences M that do not satisfy
\begin{equation}
	\label{eq:CounterexamplewI}
	\exists H > 0 \, \forall R > 0 \, \exists C > 0 \, \forall \alpha, \beta \in \N^n \, : \, M_{\alpha} R^{|\beta|} \leq C H^{|\alpha + \beta|} M_{\alpha + \beta} . 
\end{equation}
In particular, there exist log-convex weight sequence systems that do not satisfy $[\condwI]$ (take $\M_{M}$ with $M$ a weight sequence not satisfying \eqref{eq:CounterexamplewI}). In order to construct a weight sequence $M$ violating \eqref{eq:CounterexamplewI}, we consider a simplified version of the weight sequence found in  \cite[Section 6]{BJOS-lc}:  $M =(M_\alpha)_{\alpha \in \N^2}$ with $M_\alpha = e^{\max\{\alpha_1^2, \alpha_2^2\}}$ for $\alpha = (\alpha_1,\alpha_2)\in\N^2$. It is clear that $M$ is log-convex. Suppose that $M$ satisfies \eqref{eq:CounterexamplewI}. By evaluating \eqref{eq:CounterexamplewI} at  $\alpha = (j, 0)$ and $\beta = (0, j)$ for $j \in \N$, we would obtain that
$$
\exists H > 0 \, \forall R > 0 \, \exists C > 0 \, \forall j \in \N \, : \,  e^{j^2} R^j \leq C H^{2 j} e^{j^2}, 
$$
a contradiction.\end{remark}

For two weight sequence systems $\M$ and $\NN$ we write
	\begin{align*}
		\M (\subseteq) \NN & \quad \Longleftrightarrow \quad \forall \lambda \in \R_{+} \, \exists \mu \in \R_{+} \, : \, M^{\mu} \subseteq N^{\lambda} , \\
		\M \{\subseteq\} \NN & \quad \Longleftrightarrow \quad \forall \mu \in \R_{+} \, \exists \lambda \in \R_{+} \, : \, M^{\mu} \subseteq N^{\lambda} .
	\end{align*}
For two weight sequences $M$ and $N$ it holds that $\M_{M} [\subseteq] \M_{N}$ if and only if $M \preceq N$.

We associate the following weight function system to a  weight sequence system $\M$
	\[ \W_{\M} = \{ \exp \omega_{M^{\lambda}} \mid \lambda \in \R_{+} \} . \]
Given a weight sequence $M$, we  define $\W_{M} = \W_{\M_{M}} = \{ e^{\omega_M\left( \frac{\cdot}{\lambda} \right)} | \, \lambda \in \R_{+} \}$.

	\begin{lemma}
		\label{l:InclusionsWeightSeqEquivWeightFunc}
		Let $\M, \NN$ be two weight sequence systems. If $\M [\subseteq] \NN$, then $\W_{\M} [\subseteq] \W_{\NN}$. 
		If $\M$ is log-convex, the converse is also true.	
	\end{lemma}
	
	\begin{proof}
		
				 It is clear that $\M [\subseteq] \NN$ implies $\W_{\M} [\subseteq] \W_{\NN}$. 
			If $\M$ is log-convex, the converse follows from \eqref{eq:LogConvexMinorant} as $M^\mu_\alpha = (M^\mu)^{\text{lc}}_{\alpha}$ and $(N^\lambda)^{\text{lc}}_{\alpha} \leq N^\lambda_\alpha$ for any $\lambda, \mu > 0$ and $\alpha \in \N^n$.
	\end{proof}
	In the proof of the next result, we first state assertions for the Beurling case (i.e., the $(\:)$ case) followed in parenthesis by the corresponding statements for the Roumieu case ($\{\:\}$ case). We will use this convention throughout the rest of this article.
	\begin{lemma}
		\label{IandwI-1}
		Let $\M$ be a weight sequence system satisfying $[\condL]$. 
					\begin{itemize}
				\item[(i)]  If $\M$ satisfies $[\condwI]$, then $\W_{\M}$ satisfies $[\condwM]$.
				\item[(ii)]  If $\M$ satisfies $[\condI]$, then $\W_{\M}$ satisfies $[\condM]$.
			\end{itemize}
	\end{lemma}
\begin{proof} We only show (ii) as the proof of (i) is similar. Condition $[\operatorname{L}]$ implies  that 
		\begin{gather} 
		\label{Lwf}				
					\forall R > 0 \, \forall \lambda \in \R_{+} \, \exists \mu \in \R_{+} \, (\forall R > 0 \, \forall \mu \in \R_{+} \, \exists \lambda \in \R_{+}) \, \exists C > 0 \, \forall x \in \R^n \,: \\ \nonumber
					\exp \omega_{M^{\lambda}}(R x) \leq C \exp \omega_{M^{\mu}}(x).
				\end{gather}
For every $\lambda >0$ there are $\mu, \nu >0$ and $C,H >0$ (for every $\mu, \nu >0$ there are $\lambda >0$ and $C,H >0$) such that 
$$
 M^{\mu}_{\alpha}M^{\nu}_{\beta} \leq CH^{|\alpha + \beta|}M^{\lambda}_{\alpha + \beta }, \qquad \alpha, \beta \in \N^n.
$$
Hence, we obtain that for all $x,y \in \R^n$
			\begin{align*}
				\exp \omega_{M^{\lambda}}(x + y)
				&= \sup_{\alpha \in \N^{n}} \frac{\left| \sum_{\beta \leq \alpha} {\alpha \choose \beta} x^{\beta} y^{\alpha - \beta} \right| }{M^{\lambda}_{\alpha}} \\
				&\leq C \sup_{\alpha \in \N^{n}} 2^{-|\alpha|} \sum_{\beta \leq \alpha} {\alpha \choose \beta} \frac{|(2Hx)^{\beta}|}{M^{\mu}_{\beta}} \frac{|(2Hy)^{\alpha - \beta}|}{M^{\nu}_{\alpha - \beta}} \\
				&\leq C \exp \omega_{M^{\mu}}(2Hx) \exp \omega_{M^{\nu}}(2Hy) .
	\end{align*}
The result now follows from \eqref{Lwf}.
\end{proof}	
	
To conclude 
 this section, following \cite[Section 5]{R-S-CompUltradiffClass}, we introduce weight sequence systems and weight function systems generated by a weight function in the sense of \cite{B-M-T-UltradiffFuncFourierAnal}. We consider the following conditions on  a non-decreasing continuous function $\omega: [0,\infty) \to  [0,\infty)$:
	\begin{itemize}
		\item[$(\alpha)$] $\omega(2t) = O(\omega(t))$.
		\item[$(\gamma)$] $\log t = o(\omega(t))$.
		\item[$(\delta)$] $\phi : [0, \infty) \to [0, \infty)$, $\phi(x)= \omega(e^{x})$ is convex.
	\end{itemize} 
We call $\omega$ a \emph{Braun-Meise-Taylor weight function (BMT weight function)} if $\omega_{\mid[0, 1]} \equiv 0$ and $\omega$ satisfies the above conditions. In such a case, we define the \emph{Young conjugate} of $\phi$ as
	\[ \phi^{*} : [0, \infty) \to [0, \infty) , \quad \phi^{*}(y) = \sup_{x \geq 0} (yx - \phi(x)).  \]
We define $\M_{\omega} = \{ M^{\lambda}_{\omega} \mid \lambda \in \R_{+} \}$, where $M^{\lambda}_{\omega} = (\exp(\frac{1}{\lambda} \phi^{*}(\lambda |\alpha|)))_{\alpha \in \N^{n}}$. In \cite[Corollary 5.15]{R-S-CompUltradiffClass}  it is shown that $\M_{\omega}$ is a log-convex weight sequence system satisfying $[\condL]$.  For two BMT weight functions $\omega$ and $\eta$,  one has
$\M_{\omega} [\subseteq] \M_{\eta}$ if and only if $\eta(t) = O(\omega(t))$ \cite[proof of Corollary 5.17]{R-S-CompUltradiffClass}.

Given a non-decreasing continuous function $\omega: [0,\infty) \to  [0,\infty)$ tending to infinity, we define $\W_{\omega} = \{ e^{\frac{1}{\lambda}\omega(| \, \cdot \, |)} \, | \, \lambda \in \R^{+} \}$. 
Then, by \cite[Lemma 1.2]{B-M-T-UltradiffFuncFourierAnal}, $\omega$ satisfies $(\alpha)$ if and only if $\W_{\omega}$ satisfies $[\condM]$. For two  non-decreasing continuous functions with $\eta: [0,\infty) \to  [0,\infty)$ tending to infinity,
we have $\W_{\omega} [\subseteq] \W_{\eta}$ if and only if $\eta(t) = O(\omega(t))$.

\section{Statement of the main results}
\label{sec:MainResults}
In this section, we give an overview of our main results. 
Let $E$ be a solid TIBF of bounded type. For a weight sequence $M$ and  a weight function $w$ we define $E^M_w$ as the Banach space consisting of all $f \in C^\infty(\R^n)$ such that $f^{(\alpha)} w \in E$ for all $\alpha \in \N^n$ and 
	\[ \| f \|_{E, M, w} = \sup_{\alpha \in \N^{n}} \frac{\| f^{(\alpha)} w \|_{E}}{M_{\alpha}} < \infty. \]
Given a weight sequence system $\M$ and a weight function system $\W$, we define the Gelfand-Shilov type spaces
	\begin{equation}
	\label{spaces}
	E^{(\M)}_{(\W)} = \varprojlim_{\lambda \to 0^{+}} E^{M^{\lambda}}_{w^{\lambda}} , \qquad E^{\{\M\}}_{\{\W\}} = \varinjlim_{\lambda \to \infty} E^{M^{\lambda}}_{w^{\lambda}} . 
	\end{equation}
Then, $E^{(\M)}_{(\W)}$ is a Fr\'{e}chet space and $E^{\{\M\}}_{\{\W\}}$  is an $(LB)$-space.  If $\W$ satisfies $[\condwM]$, the space $E^{[\M]}_{[\W]}$ is translation-invariant, as follows by iterating $[\condwM]$. 
Given another weight sequence system $\mathfrak{N}$, we write $E^{[\M]}_{[\mathfrak{N}]} = E^{[\M]}_{[\W_{\mathfrak{N}}]}$. If  $\mathfrak{N}$ satisfies $[\condL]$, then, for any $f \in C^\infty(\R^n)$,
	\[ f \in E^{[\M]}_{[\mathfrak{N}]}  \quad \Longleftrightarrow \quad \forall \lambda > 0 ~ (\exists \lambda > 0) : \sup_{\alpha, \beta \in \N^n} \frac{\| x^\beta f^{(\alpha)} \|_{E}}{M^\lambda_{\alpha} N^\lambda_{\beta}} < \infty . \]
Let $p \in \{0\} \cup [1,\infty]$. We write $(L^p)^{[\M]}_{[\W]} =\mathcal{S}^{[\M]}_{[\W],p}$.
Given two weight sequences $M$ and $A$, we define $\mathcal{S}^{[M]}_{[A],p} = \mathcal{S}^{[\M_M]}_{[\W_A],p} = (L^p)^{[\M_M]}_{[\M_A]} $. 
Similarly, given a BMT weight function $\omega$ and a non-decreasing continuous function $\eta: [0,\infty) \to [0,\infty)$ tending to infinity, we set $\mathcal{S}^{[\omega]}_{[\eta],p} = \mathcal{S}^{[\M_\omega]}_{[\W_\eta],p}$. The spaces $\mathcal{S}^{[M]}_{[A],p}$ (for isotropic weight sequences $M$ and $A$) and $\mathcal{S}^{[\omega]}_{[\eta],p}$ were already considered in the introduction.

Fix a solid TIBF of bounded type $E$. Let $\M, \NN$ be weight sequence systems and let $\W, \V$ be weight function systems. Note that if $\M [\subseteq] \NN$ and $\W [\subseteq] \V$, then $E^{[\M]}_{[\W]} \subseteq E^{[\NN]}_{[\V]}$ continuously.
The main goal of this article is to prove the converse of this 
statement under minimal assumptions on the involved weight sequence systems and weight function systems. More precisely, we will show the following two results.  Their proofs will be given in the next section.

	\begin{theorem}
		\label{t:InclusionCharFixedWeightMatrix}
	
		Assume that $\M$ satisfies $[\condL]$ and $[\condwI]$, $\NN$ satisfies $[\condL]$, $\W$ satisfies $[\condM]$, and $\V$ satisfies $[\condwM]$. Suppose that $E^{[\M]}_{[\W]} \neq \{0\}$. 
		If $E^{[\M]}_{[\W]} \subseteq E^{[\NN]}_{[\V]}$ as sets, then $\W [\subseteq] \V$.
				\end{theorem}
	
		\begin{theorem}
		\label{t:InclusionCharFixedWeightFuncSystem}
		Assume that $\M$ is log-convex and satisfies $[\condL]$ and $[\condI]$, 
		$\NN$
		satisfies $[\condL]$ and $[\condwI]$,  and $\W$ and $\V$ satisfy $[\condwM]$. 
		Suppose that $E^{[\M]}_{[\W]} \neq \{0\}$.
		If $E^{[\M]}_{[\W]} \subseteq E^{[\NN]}_{[\V]}$ as sets, then $\M [\subseteq] \NN$.
	\end{theorem}

Theorems \ref{t:InclusionCharFixedWeightMatrix} and \ref{t:InclusionCharFixedWeightFuncSystem}  yield  the following characterization of the inclusion relations  for the type of   spaces defined in \eqref{spaces}.

	\begin{theorem}
		\label{t:InclusionChar}
		Assume that $\M$ is log-convex and satisfies $[\condL]$ and $[\condI]$, 
		$\NN$ 
		satisfies $[\condL]$ and $[\condwI]$, $\W$ satisfies $[\condM]$, and $\V$ satisfies $[\condwM]$. Suppose that $E^{[\M]}_{[\W]} \neq \{0\}$. Then, the following statements are equivalent:
			\begin{itemize}
				\item[(i)] $\M [\subseteq] \NN$ and $\W [\subseteq] \V$.
				\item[(ii)] $E^{[\M]}_{[\W]} \subseteq E^{[\NN]}_{[\V]}$ as sets.
				\item[(iii)] $E^{[\M]}_{[\W]} \subseteq E^{[\NN]}_{[\V]}$ continuously.
			\end{itemize}
	\end{theorem}

 Theorems \ref{t:InclusionCharClassicalGS}  and \ref{t:InclusionCharBMT} from the introduction are consequences of Theorem \ref{t:InclusionChar} and the properties stated in Subsection \ref{s:wss}.
 
We end this section by comparing our spaces with the ones considered in \cite{BJOS-inc, BJOS-lc}.
 	\begin{remark}
		\label{r:ComparisonBJOS}
	For a solid TIBF of bounded type $E$ and a weight sequence system $\M$ we define $E_{(\M)}$ ($E_{\{\M\}}$) as the space consisting of all functions $f \in C^\infty(\R^n)$ such that
			\[ \forall \lambda > 0 \, (\exists \lambda > 0) \, : \, \sup_{\alpha, \beta \in \N^n} \frac{\| x^\beta f^{(\alpha)} \|_E}{M^\lambda_{\alpha + \beta}} < \infty , \]
		endowed with its natural Fr\'echet space topology ($(LB)$-space topology).
		Under the assumption of $[\condL]$ on $\M$, the spaces $(L^\infty)_{[\M]}$ are the spaces considered in  \cite{BJOS-inc, BJOS-lc} (denoted there by $\mathcal{S}_{\{\mathcal{M}\}}(\R^d)$  and  $\mathcal{S}_{(\mathcal{M})}(\R^d)$).
		Consider  the condition $[\M.2]$ ($(\M_{[\text{mg}]})$ in \cite{R-S-CompUltradiffClass}) for $\M$, 
			\[ \forall \lambda \in \R_+ \, \exists \mu \in \R_+ \, (\forall \mu \in \R_+ \, \exists \lambda \in \R_+) \, \exists H > 0 \, \forall \alpha, \beta \in \N^n : M^{\mu}_{\alpha + \beta} \leq H^{|\alpha + \beta|} M^\lambda_\alpha M^\lambda_\beta . \]
		Assume that $\M$ satisfies $[\condL]$.  If $\M$ satisfies $[\condI]$, then $E^{[\M]}_{[\W_\M]} \subseteq E_{[\M]}$ continuously, while, if $\M$ satisfies $[\M.2]$, then $E_{[\M]} \subseteq E^{[\M]}_{[\W_\M]}$ continuously. Consequently, by Theorem \ref{t:InclusionChar}, the following is true: Suppose that $\M$ is a log-convex weight sequence system satisfying $[\condL]$, $[\condI]$, and $E_{[\M]} \neq \{0\}$, and that $\NN$ is a weight sequence system satisfying $[\condL]$, $[\condwI]$, and $[\M.2]$. Then, $E_{[\M]} \subseteq E_{[\NN]}$ (continuously or as sets) if and only if $\M [\subseteq] \NN$.
		In particular, under the assumption of $[\condL]$ for $\M$ and $\NN$, for $E = L^\infty$, we recover (iii) (resp.\ (i)) of \cite[Theorem 6.1]{BJOS-lc}
		where we may drop the assumptions \cite[(6.4) and (6.5)]{BJOS-lc} (resp.\ \cite[(6.1) and (6.2)]{BJOS-lc}),
			and for $\M$ weaken it to $(L^\infty)_{[\M]} \neq \{0\}$,
			while for $\NN$ we need to impose $[\condwI]$ and $[\M.2]$ (note that \cite[(6.1) and (6.4)]{BJOS-lc} implies $[\condwI]$, but $[\M.2]$ implies \cite[(6.2) and (6.5)]{BJOS-lc}). 
			\end{remark}

\section{Proofs of the main results}\label{s:proofs}
The goal of this section is to show Theorems \ref{t:InclusionCharFixedWeightMatrix} and \ref{t:InclusionCharFixedWeightFuncSystem}. We fix a solid TIBF of bounded type $E$, two weight sequence systems $\M, \NN$, and two weight function systems $\W, \V$. 
For $k \in \N$, we define the weight function system $\W_k = \{ \jb{\, \cdot \,}^{k}w^\lambda \, | \, \lambda \in \R_+ \}$.
Note that $\jb{x+y} \leq \sqrt{2} \jb{x} \jb{y}$ for all $x, y \in \R^n$. Consequently, $\mathcal{W}_k$ satisfies $[\condwM]$ ($[\condM]$) if $\mathcal{W}$ does so.

	\begin{lemma}
		\label{lemma-nt} 
		Assume that $\M$ satisfies $[\condL]$ and $[\condwI]$,  and $\W$ satisfies $[\condwM]$. Let $k \in \N$ be arbitrary. If  $E^{[\M]}_{[\W]}  \neq \{0 \}$, then $\mathcal{S}^{[\M]}_{[\W_k], \infty} \neq \{0\}$.
	\end{lemma}
	
	\begin{proof}
	We claim that there is $C >0$ such that 
		\begin{equation}
			\label{eqclaimineq}
			\|f \jb{\, \cdot \,}^{-(n+1)}\|_{L^1}\leq C\| f \|_E, \qquad f \in E.
		\end{equation}
	Before we prove this claim, let us show how it implies the result.
	Let $f \in E^{[{\M}]}_{[{\W}]} \backslash \{0\}$.
	The inequality \eqref{eqclaimineq} yields
		\begin{equation}
			\label{weightedL1}
			\forall \lambda >0 \, (\exists \lambda >0) \, : \,
			\sup_{\alpha \in \N^n} \frac{\| f^{(\alpha)} \jb{\, \cdot \,}^{-(n+1)} w^\lambda \|_{L^1}}{M^\lambda_\alpha} < \infty.
		\end{equation}
	Choose $\psi \in \mathcal{D}(\R^n)$ such that $\int_{\R^n} f(x) \psi(-x) dx =1$. Pick $\chi \in \mathcal{D}(\R^n)$ such that $\int_{\R^n} \chi (x)dx = 1$ and consider its Fourier transform $\widehat{\chi}(\xi) = \int_{\R^n} \chi(x)e^{-2\pi i \xi x} dx$.
	Set $g = (f \ast \psi) \widehat{\chi}$ and note that  $g(0) = 1$. 
		Let $k \in \N$ be arbitrary. There are $C_1, R > 0$ such that
			\[ \| \jb{\, \cdot \,}^{k + n + 1} \partial^{\beta} \widehat{\chi} \|_{L^{\infty}} \leq C_1 R^{|\beta|} , \qquad \beta \in \N^n . \]
		Suppose that $\lambda > \mu$ and $C_2,C_3,C_4 >0$ are such that: 
		$w^\lambda(x + y) \leq C_2 w^\mu(x)$ for all $x \in \R^n$, $y \in \supp \psi$;
		$\| f^{(\alpha)} \jb{\, \cdot \,}^{-(n+1)} w^\mu \|_{L^1} \leq C_3 M^\mu_{\alpha}$ for all $\alpha \in \N^n$; and
		$M^\mu_{\alpha} R^{|\beta|} \leq C_4 2^{-|\alpha + \beta|} M^\lambda_{\alpha + \beta}$ for all $\alpha, \beta \in \N^n$.
		Then,
		\begin{align*}
			|g^{(\alpha)}(x) \jb{x}^k w^\lambda(x)|
			&\leq \sum_{\beta \leq \alpha} {\alpha \choose \beta} \jb{x}^{-(n + 1)} w^\lambda(x) |(f^{(\alpha - \beta)} * \psi)(x)|  \jb{x}^{k + n + 1} | \partial^{\beta} \widehat{\chi}(x) | \\
			& \leq C_1C_2 2^{\frac{n + 1}{2}} \sum_{\beta \leq \alpha} {\alpha \choose \beta} R^{|\beta|} ([|f^{(\alpha - \beta)}| \jb{\, \cdot \,}^{-(n+1)} w^\mu] * [|\psi| \jb{\, \cdot \,}^{n+1}])(x)  \\
			& \leq C_1 C_2 C_3 2^{\frac{n + 1}{2}} \| \jb{\, \cdot \,}^{n+1} \psi \|_{L^\infty}  \sum_{\beta \leq \alpha} {\alpha \choose \beta} R^{|\beta|} M^{\mu}_{\alpha - \beta} \\
			& \leq C_1 C_2 C_3 C_4 2^{\frac{n + 1}{2}} \| \jb{\, \cdot \,}^{n+1} \psi \|_{L^\infty}  M^{\lambda}_\alpha.
		\end{align*}
By using \eqref{weightedL1}, $[\condL]$ and $[\condwI]$ for $\M$,  and $[\condwM]$ for $\W$, we find that $g \in \mathcal{S}^{[\M]}_{[\W_k], \infty}$.
	We now return to the claim \eqref{eqclaimineq}.
	Since $E \subseteq L^1_{\operatorname{loc}}(\R^n)$ continuously, there is $C >0$ such that 
		\[ \| f 1_{[0,1]^n} \|_{L^1} \leq C\|f\|_{E}, \qquad f \in E. \]
	Hence, for all $f \in E$
		\begin{align*}
			\|f \jb{\, \cdot \,}^{-(n+1)}\|_{L^1} &= \sum_{j \in \Z^n }\|f \jb{\, \cdot \,}^{-(n+1)}T_j1_{[0,1]^n}\|_{L^1}  \\
			&\leq (2(n+1))^{(n+1)/2} \sum_{j \in \Z^n} \jb{j}^{-(n+1)} \|(T_{-j}f)1_{[0,1]^n}\|_{L^1} \\
			&\leq CC_0(2(n+1))^{(n+1)/2} \sum_{j \in \Z^n} \jb{j}^{-(n+1)} \|f\|_{E},
		\end{align*}
	where $C_0$ is the constant from (A.2). This shows the claim.	
	\end{proof}

\subsection{Proof of Theorem \ref{t:InclusionCharFixedWeightMatrix}} 
We divide the proof into several steps. Given a weight function $w$, we define $E_{d,w}$ as the Banach space consisting of all $c = (c_j)_{j \in \Z^n} \in \C^{\Z^n}$ such that $(c_jw(j))_{j \in \Z^n} \in E_d$ (see Definition \ref{def:Ed}) and endow it with the norm
	\[ \|c\|_{E_{d,w}} = \|(c_jw(j))_{j \in \Z^n}\|_{E_d}. \]
We set
	\[ E_{d,(\W)} = \varprojlim_{\lambda \to 0^{+}} E_{d,w^{\lambda}} , \qquad E_{d,\{\W\}} = \varinjlim_{\lambda \to \infty} E_{d,w^{\lambda}}. \]
Then, $E_{d,(\W)}$ is a Fr\'{e}chet space and $E_{d,\{\W\}}$ is an $(LB)$-space.

\begin{proposition}
	\label{t:InclusionWeightedSpaces}
	Assume that  $\W$ and $\V$ satisfy $[\condwM]$. Then, $E_{d,[\W]} \subseteq E_{d,[\V]}$ as sets  if and only if $\W [\subseteq] \V$.
\end{proposition}

\begin{proof}
Suppose that  $\W [\subseteq] \V$. Since $E$ is solid, we obviously have that $E_{d,[\W]} \subseteq E_{d,[\V]}$ 
Conversely, suppose that $E_{d,[\W]} \subseteq E_{d,[\V]}$. De Wilde's closed graph theorem yields that the inclusion $E_{d,[\W]} \subseteq E_{d,[\V]}$ holds continuously. 
Consequently (use the Grothendieck factorization theorem in the Roumieu case), for every $\lambda > 0$ there are $\mu > 0$ and $C > 0$ (for every $\mu > 0$ there are $\lambda> 0$ and $C>0$) such that
	\[ \| c \|_{E_{d,v^\lambda}} \leq C \| c \|_{E_{d,w^{\mu}}} , \qquad \forall  c \in E_{d,(\W)} \, (\forall c \in E_{d,w^\mu}). \]
Applying this inequality to the sequences $c^{(k)} = (\delta_{k,j})_{j \in \Z^n}$, $k \in \Z^n$, we obtain that 
	\[ \frac{\| 1_{[0,1]^n}\|_E}{C_0} v^\lambda(k) \leq \| c^{(k)}  \|_{E_{d,v^\lambda}} \leq  C \| c^{(k)} \|_{E_{d,w^{\mu}}} \leq 	CC_0\| 1_{[0,1]^n}\|_E w^\mu(k), \qquad k \in \Z^n, \]
where $C_0$ is the constant from condition (A.2).
We have thus shown that
	\begin{equation}
		\label{eq:InclusionLattice}
		\forall \lambda \in \R_+ \, \exists \mu \in \R_+ \, ( \forall \mu \in \R_+ \, \exists \lambda \in \R_+) \, \exists C > 0 \, \forall k \in \Z^n \, : \, v^\lambda(k) \leq C w^\mu(k) . 
	\end{equation}
Let us now prove how this entails $\W [\subseteq] \V$.
Suppose that $\lambda, \lambda', \mu,\mu' \in \R_+$ and $C_1,C_2>0$ are such that $v^\lambda(x + y) \leq C_1 v^{\lambda'}(x)$ and $w^{\mu'}(x + y) \leq C_2 w^\mu(x)$ for all $x \in \R^n$,$y \in [-1, 1]^n$.
Moreover, assume that there is $C >0$ such that $v^{\lambda'}(k) \leq C w^{\mu'}(k)$ for all $k \in \Z^n$.
For every $x \in \R^n$   we choose $k_x \in \Z^n$ such that $x - k_x \in [0, 1]^n$.
We then find, for all $x \in \R^n$,
	\[ v^\lambda(x) \leq C_1 v^{\lambda'}(k_x) \leq C C_1 w^{\mu'}(k_x) \leq C C_1 C_2 w^{\mu}(x) . \]
Hence, as $\W$ and $\V$ both satisfy $[\condwM]$, it follows from \eqref{eq:InclusionLattice} that $\W [\subseteq] \V$.

\end{proof}

\begin{lemma}
	\label{l:Synthesis} 
	Assume that $\W$ satisfies $[\condM]$. Let $\psi \in \mathcal{S}^{[\M]}_{[\W_{n+1}], \infty}$. Then, for all $c \in  E_{d,[{\W}]}$,
		\[ R_{\psi}(c) =  \sum_{j \in \Z^{n}} c_{j} T_{j} \psi \in  E^{[\M]}_{[\W]}. \]
\end{lemma}

\begin{proof}
Let $\nu > 0$ be such that $\psi \in \mathcal{S}^{M^\nu}_{\jb{\,\cdot \,}^{n+1}w^\nu,\infty}$; this means that $\nu$  is fixed in the Roumieu case but can be taken as small as needed in the Beurling case. As $\W$ satisfies $[\condM]$, it holds that for every $\lambda > 0$ there are $\mu, \nu > 0$ and $C > 0$ (for every $\mu > 0$ there are $\lambda > 0$ and $C > 0$) such that $w^\lambda(x+y) \leq Cw^\mu(x)w^\nu(y)$ for all $x,y \in \R^n$.
For every $c \in E_{d, w^\mu}$ it holds that for all $\alpha \in \N^n$
	\[ w^{\lambda}  \sum_{j \in \Z^{n}} |c_{j} T_{j} \psi^{(\alpha)}| \leq C \sum_{j \in \Z^{n}} |c_{j} w^{\mu}(j) T_{j} (\psi^{(\alpha)} w^{\nu})| \leq C R_{|\psi^{(\alpha)}|w^\nu} ( (|c_{j}| w^{\mu}(j))_{j \in \Z^n}) . \]
Since $|\psi^{(\alpha)}|w^\nu \in C_{\jb{\, \cdot \,}^{n + 1}}$ for all $\alpha \in \N^n$, the result now follows from Lemma \ref{l:EdToE} and the fact that $E$ is solid.
\end{proof}

	\begin{lemma}
		\label{l:EvaluationMap}
	 Assume that $\V$ satisfies $[\condwM]$. For all $f \in  E^{[\NN]}_{[\V]}$ it holds that
			\[  S(f) =  (f(j))_{j \in \Z^n} \in  E_{d, [\V]}. \]
	\end{lemma}
	
	\begin{proof}
		We employ the Schwartz parametrix method. 
		Let $\chi \in \mathcal{D}(B(0,1))$ be such that $\chi \equiv 1$ on $B(0,1/2)$. For $l \in \N \setminus \{0\}$ we denote by $F_{l} \in L^1_{\operatorname{loc}}(\R^n)$ the fundamental solution of $\Delta^{l}$, where $\Delta$ is the Laplacian. Then, $\Delta^{l} (\chi F_{l}) - \delta = \varphi_{l}  \in \mathcal{D}(B(0,1))$. Note  that
			\[ f = (\Delta^{l} f) * (\chi F_{l}) - f * \varphi_l, \qquad f \in C^\infty(\R^n). \]
		Fix a sufficiently large $l$ such that $\chi F_{l} \in C_c(\R^n)$. For every $\lambda > 0$ there are $\mu > 0$ and $C > 0$ (for every $\mu > 0$ there are $\lambda > 0$ and $C > 0$) such that $v^{\lambda}(x + y) \leq C v^{\mu}(x)$ for all $x \in \R^n$, $y \in B(0,1)$. Hence, for all $f \in E^{M_\mu}_{v^\mu}$
			\[ |f v^{\lambda}| \leq C \left( |(\Delta^{l} f) v^{\mu}| * |\chi F_{l}| + |f v^{\mu}| * |\varphi_l| \right). \]
		The result now follows from Lemma \ref{l:ConvEval} and the fact that $E$ is solid
	\end{proof}

	\begin{lemma}
		\label{l:SpecificWindow1}
		 Assume that $\M$ satisfies $[\condL]$ and $[\condwI]$,  and $\W$ satisfies $[\condwM]$. Suppose that  $E^{[\M]}_{[\W]}  \neq \{0 \}$. Then, there exists  $\psi \in \mathcal{S}^{[\M]}_{[\W_{n+1}], \infty}$ such that $\psi(j) = \delta_{j, 0}$ for all $j \in \Z^{n}$.
	\end{lemma}
	
	\begin{proof}
	By Lemma \ref{lemma-nt}  there is $\varphi \in \mathcal{S}^{[\M]}_{[\W_{n+1}], \infty} \setminus \{0\}$. As the space $\S^{[\M]}_{[\W_{n+1}], \infty}$ is translation-invariant, we may assume that $\varphi(0) = 1$. Let $\chi(x) = \int_{[0, 1]^{n}} e^{-2 \pi i \xi x} d\xi$. Then, $\psi = \varphi \chi$ satisfies  $\psi(j) = \delta_{j, 0}$ for all $j \in \Z^{n}$. We now show that $\psi \in \mathcal{S}^{[\M]}_{[\W_{n+1}], \infty}$.
	Let $\mu > 0$ be such that $\varphi \in \mathcal{S}^{M^\mu}_{\jb{\,\cdot \,}^{n+1}w^\mu,\infty}$ ; this means that $\mu$ is fixed in the Roumieu case but can be taken arbitrarily small in the Beurling case.
	Since   $\M$ satisfies $[\condL]$ and $[\condwI]$,  for every $\lambda > 0$ there are $\mu < \lambda$ and $C >0$ (there are $\lambda > \mu$ and $C > 0$) such that $M^\mu_{\alpha} \cdot (2 \pi)^{|\beta|} \leq C 2^{-|\alpha + \beta|} M^\lambda_{\alpha + \beta}$ for any $\alpha, \beta \in \N^n$.
	Then, for every $\alpha \in \N^n$ and $x \in \R^n$
		\begin{align*} 
			|\psi^{(\alpha)}(x)| \jb{x}^{n + 1}  w^\lambda(x)
			&\leq \sum_{\beta \leq \alpha} {\alpha \choose \beta} |\varphi^{(\alpha - \beta)}(x)|  \jb{x}^{n + 1} w^\mu(x) (2 \pi)^{|\beta|} \\
			&\leq \|\varphi\|_{L^\infty, M^\mu, \jb{ \, \cdot \,}^{n + 1} w^\mu} \sum_{\beta \leq \alpha} {\alpha \choose \beta} M^{\mu}_{\alpha - \beta} (2 \pi)^{|\beta|} \\
			&\leq C \|\varphi\|_{L^\infty, M^\mu, \jb{ \, \cdot \,}^{n + 1} w^\mu} M^{\lambda}_{\alpha} . 
		\end{align*}
	Hence, $\psi \in \mathcal{S}^{[\M]}_{[\W_{n+1}], \infty}$.
	
	\end{proof}

	\begin{proof}[Proof of Theorem \ref{t:InclusionCharFixedWeightMatrix}]
		Choose $\psi$ as in Lemma \ref{l:SpecificWindow1}.
		 Since $E^{[\M]}_{[\W]} \subseteq E^{[\NN]}_{[\V]}$,  Lemmas \ref{l:Synthesis} and \ref{l:EvaluationMap} yield that $c = S(R_{\psi}(c)) \in E_{d, [\V]}$ for all $c \in E_{d, [\W]}$. 
		 Thus, $E_{d, [\W]} \subseteq E_{d, [\V]}$. The result now follows from Proposition \ref{t:InclusionWeightedSpaces}.
	\end{proof}

\subsection{Proof of Theorem \ref{t:InclusionCharFixedWeightFuncSystem}} 
We again divide the proof into several steps. 

	\begin{lemma}
		\label{l:ContInclusionTildeSp}
		Assume that  $\M$ and $\NN$ satisfy $[\condL]$ and $[\condwI]$. Let $k \in \N$ be arbitrary. If $E^{[\M]}_{[\W]} \subseteq E^{[\NN]}_{[\V]}$ as sets, then $E^{[\M]}_{[\W_{2k}]} \subseteq E^{[\NN]}_{[\V_{2k}]}$ as sets. 
	\end{lemma}
	
	Lemma \ref{l:ContInclusionTildeSp} follows directly from the next result.  For $k \in \N$ we denote by $\mathcal{P}_k$ the space of all polynomials on $\R^n$ of degree at most $k$.
	
	\begin{lemma}
		\label{l:TildeSpaceIsPolynomialMultiplier}
		Assume that  $\M$ satisfies $[\condL]$ and $[\condwI]$. Let $k \in \N$ be arbitrary. Then, $f \in C^\infty(\R^n)$ belongs to $E^{[\M]}_{[\W_{2k}]}$ if and only if  $fP \in E^{[\M]}_{[\W]}$ for all $P \in \mathcal{P}_{2k}$. 
	\end{lemma}
	
	\begin{proof}
	$\Rightarrow$: Since $\M$ satisfies $[\condL]$ and $[\condwI]$, it holds that for every $\lambda > 0$ there are $\mu > 0$ and $C >0$ (for every $\mu > 0$ there are $\lambda> 0$ and  $C > 0$) such that $M^\mu_\alpha \leq C2^{-|\alpha + \beta|} M^\lambda_{\alpha + \beta}$ for all $\alpha, \beta \in \N^n$. We may assume that $\mu \leq \lambda$. Let $P \in \mathcal{P}_{2k}$ be arbitrary. There is $C' > 0$ such that $|P^{(\alpha)}| \leq C' \jb{\, \cdot \, }^{2k}$ for all $\alpha \in \N^{n}$.
For every $f \in E^{M^{\mu}}_{\jb{\, \cdot \, }^{2k}w^{\mu}}$ and all $\alpha \in \N^n$
$$
|(fP)^{(\alpha)} |w^{\lambda}  \leq C' \sum_{\beta \leq \alpha} {\alpha \choose \beta} |f^{(\beta)}|   \jb{\, \cdot \, }^{2k}w^{\mu}.
$$
Since $E$ is solid, we obtain that $(fP)^{(\alpha)} w^{\lambda} \in E$ and 
$$
\| (fP)^{(\alpha)} w^{\lambda}  \|_E \leq CC' \| f \|_{E^{M^{\mu}}_{  \jb{\, \cdot \, }^{2k}w^{\mu}}} M^\lambda_{\alpha}
$$	
and thus  $fP \in E^{M^{\lambda}}_{w^{\lambda}}$, as desired.

\noindent $\Leftarrow$: For $\lambda >0$ and $j \in \N$ we write $E^{M^\lambda}_{w^\lambda, \mathcal{P}_{2j}}$ for the space consisting of all $f \in C^\infty(\R^n)$ such that $fP \in E^{M^\lambda}_{w^\lambda}$ for all $P \in \mathcal{P}_{2j}$. We claim that for all $j \in \N$ it holds that for every $\lambda >0$ there is $\mu >0$ (for every $\mu >0$ there is $\lambda >0$) such that $E^{M^\mu}_{w^\mu, \mathcal{P}_{2j}} \subseteq E^{M^\lambda}_{ \jb{\, \cdot \, }^{2j}w^\lambda}$. This implies the result as for $f \in C^\infty(\R^n)$ one has that $fP \in E^{[\M]}_{[\W]}$ for all $P \in \mathcal{P}_{2j}$ exactly means that $f \in  E^{M^\mu}_{w^\mu, \mathcal{P}_{2j}}$ for all $\mu >0$ (for some $\mu >0)$. 
We now show the claim via induction on $j$. The case $j =0$ is trivial. Suppose that the claim is true for $j$ and let us verify it for $j +1$. 
The induction hypothesis and the fact that $\M$ satisfies   $[\condL]$ and $[\condwI]$ imply that for all $\lambda > 0$ there are $\mu,\nu >0$ and $C >0$ (for every $\mu>0$ there are $\nu,\lambda >0$ and $C >0$) such that  $E^{M^\mu}_{w^\mu, \mathcal{P}_{2j}} \subseteq E^{M^\nu}_{ \jb{\, \cdot \, }^{2j}w^\nu}$ and $M^\nu_\alpha \leq C 3^{-|\alpha + \beta|} M^\lambda_{\alpha + \beta}$ for all $\alpha, \beta \in \N^n$. We may assume that $\mu \leq \nu \leq \lambda$. 
Note that $\jb{\, \cdot \, }^{2j+2} \in \mathcal{P}_{2j+2}$ and that there is $C' > 0$ such that 
	\[  |(\jb{\, \cdot \, }^{2j+2})^{(\alpha)}| \leq C'\jb{\, \cdot \, }^{2j+2 - |\alpha|}, \qquad |\alpha| \leq 2j+2. \]
Hence, for every $f \in E^{M^\mu}_{w^\mu, \mathcal{P}_{2j+2}}$ and all $\alpha \in \N^n$ 
	\begin{align*}
		| f^{(\alpha)}|  \jb{\, \cdot \, }^{2j+2}w^\lambda &\leq  |(f  \jb{\, \cdot \, }^{2j+2})^{(\alpha)}|w^\mu + C'\sum_{\substack{\beta \leq \alpha \\ 1 \leq |\alpha - \beta| \leq 2j+2}} {\alpha \choose \beta} |f^{(\beta)}| \jb{\, \cdot \, }^{2j+2 - |\alpha -\beta|} w^\nu \\
		&\leq  |(f  \jb{\, \cdot \, }^{2j+2})^{(\alpha)}|w^\mu + C'\sum_{\beta \leq \alpha} {\alpha \choose \beta} |f^{(\beta)}| \jb{\, \cdot \, }^{2j+1} w^\nu.  \\
	\end{align*}
Furthermore, for each $\beta \in \N^n$ and for all $x \in \R^n$,
	\begin{align*}
		|f^{(\beta)}(x)| \jb{x}^{2j+1}  &\leq |f^{(\beta)}(x)| \jb{x}^{2j} + \sum_{1 \leq l \leq n } |f^{(\beta)}(x)x_l | \jb{x}^{2j}  \\
		&\leq |f^{(\beta)}(x)| \jb{x}^{2j} + \sum_{1 \leq l \leq n } |(f(x)x_l)^{(\beta)} | \jb{x}^{2j} + \sum_{\substack{1 \leq l \leq n \\ \beta_l \neq 0}}\beta_l |f^{(\beta - e_l)}(x) | \jb{x}^{2j}. 
	\end{align*}
The induction hypothesis and the fact that $E$ is solid therefore imply that $f^{(\alpha)}  \jb{\, \cdot \, }^{2j+2}w^\lambda \in E$ for all $\alpha \in \N^n$ and
	$$
	\| f^{(\alpha)}  \jb{\, \cdot \, }^{2j+2}w^\lambda \|_E \leq C'' M^\lambda_\alpha,
	$$
where
	$$
	C'' =  \|f  \jb{\, \cdot \, }^{2j+2}\|_{E^{M^\mu}_{w^\mu}} + CC'(n+1)\| f \|_{E^{M^\nu}_{\jb{\, \cdot \, }^{2j}w^\nu}} + CC'\sum_{1\leq l \leq n} \| fx_l \|_{E^{M^\nu}_{\jb{\, \cdot \, }^{2j}w^\nu}},
	$$
and thus $f \in E^{M^\lambda}_{\jb{\, \cdot \, }^{2j+2}w^\lambda}$. This shows the claim. 
\end{proof}	
		
Given  a weight sequence $M$, we define $E^{M}_{\operatorname{per}}$ as the Banach space consisting of all $\Z^{n}$-periodic $f \in C^\infty(\R^n)$ such that $f^{(\alpha)}1_{[0, 1]^{n}}  \in E$ for all $\alpha \in \N^n$ and 
	$$
	\|f\|_{E^{M}_{\operatorname{per}}} = \sup_{\alpha \in \N^{n}} \frac{\| f^{(\alpha)}1_{[0, 1]^{n}}\|_{E}}{M_{\alpha}} < \infty. 
	$$
We set
	\[ E^{(\M)}_{\operatorname{per}}  = \varprojlim_{\lambda \to 0^{+}} E^{M^{\lambda}}_{\operatorname{per}} , \qquad E^{\{\M\}}_{\operatorname{per}}= \varinjlim_{\lambda \to \infty} E^{M^{\lambda}}_{\operatorname{per}} . \]
Then, $E^{(\M)}_{\operatorname{per}}$ is a Fr\'echet space and  $E^{\{\M\}}_{\operatorname{per}}$ is an $(LB)$-space.

\begin{proposition}
	\label{t:InclusionPeriodicFunc}
Assume that $\M$ is log-convex and satisfies $[\condwI]$ and that $\NN$ satisfies $[\condwI]$. Then, $E^{[\M]}_{\operatorname{per}} \subseteq E^{[\NN]}_{\operatorname{per}}$ as sets  if and only if $\M [\subseteq] \NN$. 	 
\end{proposition}
	
\begin{proof}
Clearly, $\M [\subseteq] \NN$ implies that $E^{[\M]}_{\operatorname{per}} \subseteq E^{[\NN]}_{\operatorname{per}}$. 
Now suppose that $E^{[\M]}_{\operatorname{per}} \subseteq E^{[\NN]}_{\operatorname{per}}$. 
De Wilde's closed graph theorem yields that the inclusion $E^{[\M]}_{\operatorname{per}} \subseteq E^{[\NN]}_{\operatorname{per}}$ holds continuously. 
Consequently (making use again of the Grothendieck factorization theorem in the Roumieu case), for every $\lambda > 0$ there are $\mu > 0$ and $C > 0$ (for every $\mu > 0$ there are $\lambda> 0$ and  $C > 0$) such that
	$$
		\| f \|_{E^{N^\lambda}_{\operatorname{per}}}  \leq C \| f \|_{E^{M^\mu}_{\operatorname{per}}}  , \qquad \forall  f \in E^{(\M)}_{\operatorname{per}}  \, (\forall f \in E^{M^\mu}_{\operatorname{per}}).
	$$
Taking the functions $f_k(x) = e^{2 \pi i kx }$, $k \in \Z^n$, in this inequality, we obtain, by the solidity of $E$, that for all $k \in \Z^n$
	$$
		\|1_{[0, 1]^{n}}\|_{E}	\exp \omega_{N^\lambda}(2\pi k) = \| f_k \|_{E^{N^\lambda}_{\operatorname{per}}}  \leq  C \| f_k \|_{E^{M^\mu}_{\operatorname{per}}} = C\|1_{[0, 1]^{n}}\|_{E}	 \exp \omega_{M^\mu}(2\pi k). 
	$$
Since both $\W_\M$ and $\W_\NN$ satisfy $[\condwM]$ (Lemma \ref{IandwI-1}(i)), by using a similar argument as in the proof of Proposition \ref{t:InclusionWeightedSpaces}, the previous inequality implies that  $\W_\M [\subseteq] \W_\NN$. The result now follows from  Lemma \ref{l:InclusionsWeightSeqEquivWeightFunc}.
\end{proof}

\begin{lemma}\label{hulp1}
Assume that $\M$ satisfies $[\condL]$ and $[\condI]$. Let $k \in \N$ be arbitrary and let $\psi \in \mathcal{S}^{[\M]}_{[\W_{k+n+1}], \infty}$. Then, $L_\psi(f) = \psi f \in E^{[\M]}_{[\W_k]}$ for all  $f \in E^{[\M]}_{\operatorname{per}}$. 
\end{lemma}

\begin{proof}
Let $\nu > 0$ be such that $\psi \in \mathcal{S}^{M^\nu}_{\jb{\,\cdot \,}^{k+n+1}w^\nu,\infty}$; this means that $\nu$  is fixed in the Roumieu case but can be taken as small as needed in the Beurling case. 
As $\M$ satisfies $[\condL]$ and $[\condI]$, for every $\lambda > 0$ we can find $\mu, \nu > 0$ and $C > 0$ (for every $\mu > 0$ there are $\lambda > 0$ and $C > 0$) such that $M^{\mu}_{\alpha}M^\nu_{\beta} \leq C2^{-|\alpha + \beta|}  M^{\lambda}_{\alpha + \beta}$ for all $\alpha, \beta \in \N^n$. 
We may assume that $\nu \leq \lambda$. For each $f \in  E^{M^{\mu}}_{\operatorname{per}}$ and all $\alpha \in \N^n$,
\begin{align*}
&| (\psi f)^{(\alpha)} | \jb{\, \cdot \,}^kw^{\lambda}  \\
&\leq  \sum_{\beta \leq \alpha} {\alpha \choose \beta} \sum_{j \in \Z^{n}} | \psi^{(\alpha -\beta)} | \jb{\, \cdot \,}^kw^{\nu} |f^{(\beta)}| T_j 1_{[0,1]^n} \\
&\leq (2(n+1))^{(n+1)/2}   \sum_{\beta \leq \alpha} {\alpha \choose \beta}  | \psi^{(\alpha - \beta)} | \jb{\, \cdot \,}^{k+n+1}w^{\nu} \sum_{j \in \Z^{n}} |f^{(\beta)}| T_j 1_{[0,1]^n} \jb{j}^{-(n+1)}.
\end{align*}
Since $f$ is $\Z^n$-periodic, we have
$$
 \sum_{j \in \Z^{n}} \|f^{(\beta)} T_j 1_{[0,1]^d}\|_E \jb{j}^{-(n+1)} \leq C_0   \|f^{(\beta)} 1_{[0,1]^d} \|_E \sum_{j \in \Z^{n}}  \jb{j}^{-(n+1)},
$$
where $C_0$ is the constant from (A.2). The fact that $E$ is solid therefore implies that $(\psi f)^{(\alpha)} \jb{\, \cdot \,}^kw^{\lambda}  \in E$ and (with $C' = C_0 (2(n+1))^{(n+1)/2}\sum_{j \in \Z^{n}}  \jb{j}^{-(n+1)}$)
\begin{align*}
\| (\psi f)^{(\alpha)}  \jb{\, \cdot \,}^kw^{\lambda} \|_E &\leq C' \| \psi\|_{ \mathcal{S}^{M^\nu}_{\jb{\,\cdot \,}^{k+n+1}w^\nu,\infty}} \| f\|_{E^{M^{\mu}}_{\operatorname{per}}} \sum_{\beta \leq \alpha} {\alpha \choose \beta} M^\mu_\beta M^\nu_{\alpha - \beta} \\
&\leq  CC' \| \psi\|_{ \mathcal{S}^{M^\nu}_{\jb{\,\cdot \,}^{k+n+1}w^\nu,\infty}} \| f\|_{E^{M^{\mu}}_{\operatorname{per}}} M^{\lambda}_{\alpha}
\end{align*}		
and thus $	\psi f \in E^{M^\lambda}_{\jb{\,\cdot \,}^{k}w^\lambda}$. 
This completes the proof of the lemma.	
\end{proof}		

\begin{lemma}\label{hulp2}
For all $f \in {E}^{[\NN]}_{[\V_{n+1}]}$,
	\[ \Pi(f) = \sum_{j \in \Z^{n}} T_{j} f \in  E^{[\NN]}_{\operatorname{per}}.  \]		
\end{lemma}

\begin{proof} Let $\lambda >0$ be arbitrary. For all $f \in  {E}^{N^{\lambda}}_{\jb{\,\cdot \,}^{n+1} v^{\lambda}}$ and $\alpha \in \N^n$,
	$$
		1_{[0,1]^n}  \sum_{j \in \Z^{n}} |T_{j} f^{(\alpha)}| \leq (2(n+1))^{(n+1)/2} \sum_{j \in \Z^n}  |T_j( f^{(\alpha)} \jb{\, \cdot \,}^{n+1})| \jb{j}^{-(n+1)}. 
	$$
Moreover,
	$$
		\sum_{j \in \Z^n} \| T_j( f^{(\alpha)} \jb{\, \cdot \,}^{n+1})\|_E \jb{j}^{-(n+1)} \leq C_0\| f^{(\alpha)} \jb{\, \cdot \,}^{n+1} v^\lambda \|_E \sum_{j \in \Z^n}\jb{j}^{-(n+1)},
	$$
where $C_0$ is the constant from (A.2). Hence, as  $E$ is solid, we obtain that $\sum_{j \in \Z^{n}} T_{j} f \in E^{N^\lambda}_{\operatorname{per}}$, as claimed.
\end{proof}		

\begin{lemma}
\label{l:SpecificWindow2}
 Assume that $\M$ satisfies $[\condL]$ and $[\condI]$,  and $\W$ satisfies $[\condwM]$. Suppose that  $E^{[\M]}_{[\W]}  \neq \{0 \}$.  Let $k \in \N$ be arbitrary. Then, there exists  $\psi \in \mathcal{S}^{[\M]}_{[\W_{k}], \infty}$ such that $\Pi (\psi) = \sum_{j \in \Z^{n}} T_{j} \psi \equiv 1$.	
\end{lemma}

\begin{proof}
We may assume that $k \geq n+1$. 
We start by showing that $\S^{[\M]}_{[\W_k], \infty}$ is closed under pointwise multiplication. Suppose that $\lambda \geq \mu$ and $C >0$ are such that $M^{\mu}_{\alpha} M^{\mu}_{\beta} \leq C 2^{-|\alpha + \beta|} M^{\lambda}_{\alpha + \beta}$ for all $\alpha, \beta \in \N^n$. Then, for all $\psi, \rho \in \S^{M^\mu}_{w^\mu \jb{\, \cdot \,}^k, \infty}$ and $\alpha \in \N^n$,
	\begin{align*} 
		|(\psi \rho)^{(\alpha)}| w^\lambda \jb{\, \cdot \,}^k
		&\leq \sum_{\beta \leq \alpha} {\alpha \choose \beta} |\psi^{(\alpha - \beta)}| w^\mu \jb{\, \cdot \,}^k |\rho^{(\beta)}| w^\mu \jb{\, \cdot \,}^k \\
		&\leq \| \psi \|_{L^\infty, M^\mu, w^\mu \jb{\, \cdot \,}^k} \| \rho \|_{L^\infty, M^\mu, w^\mu \jb{\, \cdot \,}^k} \sum_{\beta \leq \alpha} {\alpha \choose \beta} M^\mu_{\alpha - \beta} M^\mu_{\beta}  \\
		&\leq C \| \psi \|_{L^\infty, M^\mu, w^\mu \jb{\, \cdot \,}^k} \| \rho \|_{L^\infty, M^\mu, w^\mu \jb{\, \cdot \,}^k} M^\lambda_\alpha ,
	\end{align*}
so that $\psi \rho \in \S^{M^\lambda}_{w^\lambda \jb{\, \cdot \,}^k, \infty}$. Since $\M$ satisfies $[\condL]$ and $[\condI]$, we obtain that  $\S^{[\M]}_{[\W_k], \infty}$ is closed under pointwise multiplication.
 By Lemma \ref{lemma-nt}  there is $\varphi \in \mathcal{S}^{[\M]}_{[\W_{k}], \infty} \setminus \{0\}$. Then, $|\varphi|^2 = \varphi \overline{\varphi} \in \mathcal{S}^{[\M]}_{[\W_{k}], \infty} \setminus \{0\}$. Then, $\varphi_0 =  |\varphi|^2 / \| |\varphi|^2 \|_{L^1} \in \mathcal{S}^{[\M]}_{[\W_{k}], \infty}$ and $\int_{\R^n} \varphi_0(x) dx = 1$. Set 
	$$
		\psi(x) = \int_{[0, 1]^n} \varphi_0(x-t) dt
	$$
so that $\sum_{j \in \Z^{n}} T_{j} \psi \equiv 1$. 
As  $\W$ satisfies $[\condwM]$, we  have that $\psi \in \mathcal{S}^{[\M]}_{[\W_{k}], \infty}$.
\end{proof}
	
\begin{proof}[Proof of Theorem \ref{t:InclusionCharFixedWeightFuncSystem}] 
Fix $k \in \N$ such that $2k \geq n+1$.
By Lemma \ref{l:SpecificWindow2} there is  $\psi \in \mathcal{S}^{[\M]}_{[\W_{2k + n+1}],\infty}$ such that $\sum_{j \in \Z^{n}} T_{j} \psi \equiv 1$. 
Then for any $\Z^n$-periodic periodic function $f$ in $E$ we have
	\[ \Pi(L_\psi(f)) = \Pi(\psi f) = \sum_{j \in \Z^n} T_j(\psi f) = f  \sum_{j \in \Z^n} T_j \psi = f . \]
Now, Lemma \ref{l:ContInclusionTildeSp} implies that $E^{[\M]}_{[\W_{2k}]} \subseteq E^{[\NN]}_{[\V_{2k}]}$.  In view of the latter inclusion,  Lemmas \ref{hulp1} and \ref{hulp2} yield that $f = \Pi(L_{\psi} (f)) \in E^{[\NN]}_{\operatorname{per}}$ for all $f \in E^{[\M]}_{\operatorname{per}}$. 
Thus, $E^{[\M]}_{\operatorname{per}} \subseteq E^{[\NN]}_{\operatorname{per}}$. The result now follows from Proposition \ref{t:InclusionPeriodicFunc}.
\end{proof}

\subsection*{Acknowledgement}
We thank the anonymous referee for helpful comments to improve the paper, in particular for suggesting the counterexample in Remark \ref{r:Counterexample}.

\end{document}